\documentclass[11pt, a4paper]{article}
\usepackage{amssymb,amsmath, wasysym, graphicx, epsfig, setspace, wrapfig, comment, pgfplots, fancyhdr, theorem, mathtools}
\usepackage{latexsym,dsfont,amsfonts,amsmath,amssymb}
\usepackage{pstricks,epsfig,psfrag,color}
\usepackage{caption}
\usepackage{subcaption}
\usepackage{soul}
\usepackage{mathrsfs}   

\newcommand{\pd}[3]{\frac{\partial ^{#1} #2}{\partial #3}}
\newcommand{\dd}[3]{\dfrac{\rd^{#1}#2}{\rd#3^{#1}}}

\newcommand{\nrm}[2]{\ensuremath{\left\|#1\right\|_{#2}}}

\newcommand{\sumu}[0]{\sum_{\setu \subseteq \{1:s\}}}

\newcommand{\produ}[0]{\prod_{j \in \setu}}
\newcommand{\dottimes}[0]{.\hspace*{-2pt}*}

\newcommand{\N}[0]{\mathbb{N}}

\newcommand{\U}[0]{\mathbb{U}}

\pgfplotsset{compat=newest}
\pgfplotsset{plot coordinates/math parser=false}
\newlength\figureheight
\newlength\figurewidth

\newcommand{\W}[0]{\ensuremath{\mathcal{W}}}

\makeatletter
\newcommand{\vast}{\bBigg@{4}}
\newcommand{\Vast}{\bBigg@{5}}
\makeatother

\newcommand{\wce}[1]{\ensuremath{e_{#1}}}
\newcommand{\wcesh}[1]{\ensuremath{\wce{#1}^\mathrm{sh}}}
\newcommand{\Qsh}[1]{\ensuremath{Q_{#1}^\mathrm{sh}}}

\usepackage[normalem]{ulem}

\newtheorem{theorem}{Theorem}
\newtheorem{lemma}[theorem]{Lemma}

\newenvironment{proof}{\begin{trivlist}
    \item[\hskip\labelsep{\bf Proof.}]}{$\hfill\Box$\end{trivlist}}
{\theoremstyle{plain} \theorembodyfont{\rmfamily}
\newcounter{rem}
\newtheorem{remark}[rem]{Remark}

\newcounter{algo}
\newtheorem{algorithm}[algo]{Algorithm}

\newcommand{\bsDelta}{{\boldsymbol{\Delta}}}

\newcommand{\bsb}{{\boldsymbol{b}}}
\newcommand{\bsB}{{\boldsymbol{B}}}

\newcommand{\bsgamma}{{\boldsymbol{\gamma}}}

\newcommand{\bsp}{{\boldsymbol{p}}}

\newcommand{\bst}{{\boldsymbol{t}}}

\newcommand{\bsx}{{\boldsymbol{x}}}

\newcommand{\bsz}{{\boldsymbol{z}}}

\newcommand{\bszero}{{\boldsymbol{0}}}
\newcommand{\bsone}{{\boldsymbol{1}}}

\newcommand{\bsOmega}{\boldsymbol{\Omega}}

\newcommand{\rd}{\mathrm{d}}

\newcommand{\bbN}{\mathbb{N}}
\newcommand{\bbE}{\mathbb{E}}

\newcommand{\calO}{\mathcal{O}}

\newcommand{\setu}{{\mathrm{\mathfrak{u}}}}

\newcommand{\mask}[1]{{}}

\newcommand{\Ber}{\mathscr{B}}

\definecolor{darkred}{RGB}{139,0,0}
\definecolor{darkgreen}{RGB}{0,100,0}
\definecolor{darkmagenta}{RGB}{170,0,120}
\definecolor{darkpurple}{RGB}{110,0,180}
\definecolor{darkblue}{RGB}{40,0,200}
\definecolor{darkbrown}{rgb}{0.75,0.40,0.15}

\newcommand{\be}{\begin{equation}}
\newcommand{\ee}{\end{equation}}
\newcommand{\bea}{\begin{eqnarray}}
\newcommand{\eea}{\end{eqnarray}}
\newcommand{\beas}{\begin{eqnarray*}}
\newcommand{\eeas}{\end{eqnarray*}}

\def\r2p{{\sqrt{2\pi}}}
\graphicspath{{./Figures/}{./}}

\title{Hiding the weights -- CBC black box algorithms with a guaranteed error bound}
\date{\today}
\author{Alexander D. Gilbert, Frances Y. Kuo and Ian H. Sloan}

\begin{document}
\maketitle

\begin{abstract}
The component-by-component (CBC) algorithm is a method for constructing
good generating vectors for lattice rules for the efficient
computation of high-dimensional integrals in the ``weighted'' function
space setting introduced by Sloan and Wo\'zniakowski. The ``weights''
that define such spaces are needed as inputs into the CBC
algorithm, and so a natural question is, for a given problem how does
one choose the weights? This paper introduces two new 
CBC algorithms which, given bounds on 
the mixed first derivatives of the integrand, produce a randomly 
shifted lattice rule with a guaranteed bound on the root-mean-square 
error. This alleviates the need for the user to specify the weights.
We deal with ``product weights'' and ``product and order dependent (POD) weights''.
Numerical tables compare the two algorithms under various assumed bounds
on the mixed first derivatives, and provide rigorous upper
bounds on the root-mean-square integration error.
\end{abstract}

\section{Introduction}
Our aim in the current work is the efficient and relatively painless
numerical computation of high-dimensional integrals of the form
\begin{align*} 
I_sf \,\coloneqq\, \int_{[0,1]^s} f(\bsx) \,\rd \bsx \, ,
\end{align*}
where $[0, 1]^s\coloneqq [0, 1]\times \cdots \times [0, 1]$ denotes
the $s$-dimensional unit cube. A quasi-Monte Carlo (QMC) approximation of
the above integral is an equal-weight quadrature rule
\begin{align*} 
Q_{n,s}(P_n)f \,\coloneqq\, \frac{1}{n} \sum_{k = 0}^{n - 1} f(\bst_k) \, ,
\end{align*}
where the quadrature points, $P_n \coloneqq \{\bst_k\}_{k = 0}^{n - 1}$,
are chosen deterministically from $[0, 1]^s$. The setting for the error
analysis of such QMC approximations, introduced by Sloan and
Wo\'zniakowski \cite{SW98}, assumes that the integrand $f$ belongs to
some $s$-variate weighted function space $W_{s,\bsgamma}$, where in
the original formulation each variable has an associated ``weight''
parameter $\gamma_i
> 0$ whose size describes the importance of~$x_i$. Weights of this kind
$\bsgamma = \{\gamma_i\}_{i = 1}^\infty$ are incorporated into
$W_{s,\bsgamma}$ through the weighted norm $\nrm{\cdot}{s, \bsgamma}$ and
are nowadays referred to as \emph{product weights}, see
Section~\ref{sec:qmc} for more details. In Section~\ref{sec:qmc} we also
consider more general weights, that is, a parameter $\gamma_\setu$ is
allowed for each subset $\setu\subseteq\{1,\ldots,s\}$, but in the
Introduction we shall concentrate on product weights.

This paper introduces two new variants of the component-by-component (CBC)
algorithm that, given a bound on the norm of the integrand, choose not
only the QMC points but also the weights, with a view to minimising
the error of the QMC approximation. Although still playing an important role
in our constructions, the weight parameters no longer need to be chosen by the 
practitioner because this choice is handled automatically inside the algorithms.

More precisely, in such a weighted function space $W_{s,\bsgamma}$, the
\emph{worst-case error} of $Q_{n,s}$ over the unit ball of
$W_{s,\bsgamma}$, is defined by
\begin{align*} 
e_{n, s, \bsgamma}(P_n) \,\coloneqq\,
\sup_{\nrm{f}{s, \bsgamma} \leq 1} \left|I_sf - Q_{n,s}(P_n)f\right| \, ,
\end{align*}
from which it follows by linearity that the error of a QMC
approximation satisfies
\begin{align}
\label{eq:e*||f||}
\left|I_sf - Q_{n,s}(P_n)f\right| \,\leq\, e_{n, s, \bsgamma}(P_n) \nrm{f}{s, \bsgamma} \, .
\end{align}
This error bound is convenient because of its separation into the
product of two factors, one which depends only on the quadrature points
and the other which depends only on the integrand. A key aspect of the
current work is that both the worst-case error and the norm depend on the
weights.

In this paper the user is assumed to have information about the 
norm $\nrm{f}{s, \bsgamma}$ (defined in \eqref{eq:nrm_un} below) in 
the form of estimates of the parameters $B_\ell$ and $b_i$ in the 
following assumption:

\textit{\textbf{Standing Assumption:} For two sequences of positive
real numbers $(B_\ell)_{\ell = 1}^s$ and $(b_i)_{i = 1}^s$, we assume that
the mixed first derivatives of the integrand satisfy the following family
of upper bounds, for each $\setu \subseteq \{1:s\}$:
\begin{align}\label{eq:L2bnd}
\int_{[0, 1]^{|\setu|}}\left(\int_{[0, 1]^{s - |\setu|}} \pd{|\setu|}{}{\bsx_\setu} f(\bsx) \, \rd \bsx_{-\setu} \right)^2 \rd \bsx_\setu \,\leq\, B_{|\setu|} \prod_{j\in \setu} b_j^2 \, .
\end{align}
}
Here $\{1:s\}$ is shorthand for $\{1, 2, \ldots, s\}$, $\bsx_\setu = \{x_j
\}_{j \in \setu}$ are the active variables,
$\bsx_{-\setu}=\{x_j\}_{j\in\{1:s\}\setminus\setu}$ are the inactive
variables, and $\partial^{|\setu|}/\partial\bsx_\setu$ is the first-order,
mixed partial derivative with respect to $\bsx_\setu$.

Bounds of the form given in \eqref{eq:L2bnd}, together with explicit values of
$B_\ell$ and $b_i$, have been found for a particular PDE problem in several recent
papers, including \cite{KSS12}.
 
The CBC algorithm, first invented by Korobov \cite{Kor59}, and
rediscovered in \cite{SR01,SKJ02b}, is an efficient method of constructing
``good'' QMC point sets such as lattice rules
\cite{Nied92,SJ94}. The idea behind the construction is to work
through each dimension $i = 1, 2, \ldots, s$ sequentially, choosing the
$i$th component of the rule by minimising the worst-case error in that
dimension while all previous components remain fixed. Because the
worst-case error depends explicitly on the weights, the CBC algorithm
requires the weights as inputs. Thus, the weights are not only useful
for theory, but also as a practical necessity. The relevant aspects of
the CBC algorithm and general QMC theory will be presented in
Section~\ref{sec:qmc}, however, for a more complete introduction to the
concepts above see \cite{DKS13}.

The CBC construction has the virtue that, as was shown in \cite{Dick04,
Kuo03}, the worst-case error of the resulting QMC approximation converges
to zero at a rate that, depending on the weights, can be arbitrarily
close to $n^{-1}$, with a constant that can be independent of $s$. From
\cite{Dick04, Kuo03}, the root-mean-square error of a CBC generated 
``randomly shifted lattice rule'' approximation of $I_sf$ (in the special case 
of prime $n$ and product weights) is bounded above by
\begin{align}
\label{eq:errbndintro}
\left(\frac{1}{n - 1} \prod_{i = 1}^s
\left(1 + \gamma_i^\lambda \frac{2\zeta(2\lambda)}{(2\pi^2)^\lambda}\right)\right)^\frac{1}{2\lambda}
\nrm{f}{s, \bsgamma} \quad \text{for all } \lambda \in \left(\tfrac{1}{2}, 1\right] \, ,
\end{align}
where $\zeta(x) = \sum_{k=1}^\infty k^{-x}$ is the Riemann zeta
function.

Until recently the choice of weights was generally \emph{ad
hoc}, but in the paper \cite{KSS12} a new principle was used to determine
weights for a particular problem: having estimated an upper bound on the
norm of the integrand (which like the worst-case error depends on the
weights), those authors chose weights that minimise an upper bound on the
error \eqref{eq:errbndintro}. (Note that \cite{KSS12} dealt with
a specific problem of randomly shifted lattice rules applied to PDEs
with random coefficients, however the strategy of that paper can easily be
applied to other problems.) The result is a family of weight sequences
indexed by the parameter $\lambda$, where $\lambda$ affects the
theoretical rate of convergence. The fact that $\lambda$ must be chosen 
by the user is a major drawback of the strategy in \cite{KSS12}. 
One option would be to take $\lambda$ as close to $\frac{1}{2}$ as possible 
to ensure a good convergence rate, however, because of the occurrence of 
the zeta function $\zeta(2\lambda)$, the constant goes to $\infty$ as $\lambda
\rightarrow \frac{1}{2}$. A good rate of convergence does not help for a
fixed value of $n$ if the constant becomes too large. To obtain the best
bound a delicate balance between the two factors in
\eqref{eq:errbndintro} is needed.

Another drawback of the method used in \cite{KSS12} is that the bound
\eqref{eq:errbndintro} is often a crude overestimate. The first algorithm
we introduce in this paper, the \textbf{double CBC (DCBC) algorithm},
counters this while at the same time removing the need to choose
$\lambda$ by dealing with the exact ``shift-averaged'' worst-case error
(see Section~\ref{sec:qmc}), rather than the upper bound given by the
first factor in \eqref{eq:errbndintro}. For the case of product weights,
at step $i$ of the DCBC algorithm, after fixing the component of the
lattice rule, the weight $\gamma_i$ is chosen so as to minimise a bound
on the error in the current dimension. An advantage of this method is that
the choice of weight in each dimension adds virtually no extra
computational cost to the algorithm.

The second algorithm we propose begins with the upper bound
\eqref{eq:errbndintro}, and hence the family of weights indexed by
$\lambda$ obtained following the strategy in \cite{KSS12}. To choose the
``best'' $\lambda$, and in turn the ``best'' weights, an iteration of the
CBC algorithm with respect to $\lambda$ is employed to minimise
heuristically a bound on the approximation error. Because of this
iteration process it is called the \textbf{iterated CBC (ICBC) algorithm}.

The philosophy of both algorithms is to concentrate on reducing
the guaranteed error bounds.
This paradigm represents a shift away from the usual focus on
the best rate of convergence. It is particularly useful when
dealing with problems where function evaluations are highly expensive,
such as those where QMC methods have been applied to PDEs with random
coefficients \cite{KSS12}, for which there is often a practical
limit on the number of quadrature points $n$.

So far we have discussed only product weights, however, both algorithms
can be extended to cover a more general form of weights called ``POD
weights'' (see Section~\ref{sec:qmc}).

The structure of the paper is as follows. Section~\ref{sec:qmc} provides a
brief review of the relevant aspects of QMC theory. Details on our two new
algorithms are given in Section~\ref{sec:var}. In Section~\ref{sec:num_cbc} 
we give tables for the guaranteed error bounds resulting from the two
algorithms, under various assumptions on the parameters $B_\ell$ and $b_i$
in \eqref{eq:L2bnd}. The examples show that there are different situations
where each of the algorithms outperforms the other, thus it is not 
possible to say that one algorithm always outperforms the other.

\section{Relevant quasi-Monte Carlo theory}\label{sec:qmc}

Here we fix notation and briefly outline the relevant aspects of QMC
theory, including weighted function spaces for error analysis, randomly
shifted lattice rules and the CBC construction. For a more comprehensive
overview the reader is referred to the review papers \cite{DKS13,
KSS12ANZIAM} or the book \cite{DP10}.

\subsection{Weighted Sobolev spaces and randomisation}

Given a collection of positive real numbers, $\bsgamma = \{\gamma_\setu\}$
where $\setu$ denotes a finite subset of  $\N = \{1, 2, \ldots, \}$,
let $W_{s, \bsgamma}$ be the $s$-dimensional weighted Sobolev space with
unanchored norm
\begin{align} \label{eq:nrm_un}
\nrm{f}{s, \bsgamma}^2 \,=\, \sum_{\setu\subseteq \{1:s\}} \frac{1}{\gamma_\setu}
\int_{[0,1]^{|\setu|}} \left(\int_{[0,1]^{s - |\setu|}} \pd{|\setu|}{}{\bsx_\setu}
f(\bsx)\,\rd\bsx_{-\setu}\right)^2\,\rd\bsx_\setu \, .
\end{align}

In practice it is difficult to work with general weights $\gamma_\setu$
and so often weights with some inherent structure are used. The three most
common forms are \emph{product weights} where $\gamma_\setu = \prod_{j \in
\setu}\gamma_j$ for some sequence $1\geq\gamma_1 \geq \gamma_2 \geq \cdots
> 0$; \emph{order dependent weights} where each weight depends only on the
cardinality of the set, $\gamma_\setu = \Gamma_{|\setu|}$, for a sequence
of positive real numbers $\Gamma_0:=1, \Gamma_1, \Gamma_2, \dotsc$;
and \emph{product and order dependent \textnormal{(}POD\textnormal{)}
weights} which are a hybrid of the previous two, with $\gamma_\setu =
\Gamma_{|\setu|} \prod_{j \in \setu} \gamma_j$.

Given a random shift $\bsDelta$ uniformly distributed on $[0, 1]^s$,
and point set $P_n = \{\bst_k\}_{k = 0}^{n - 1}$, the randomly
shifted point set $(P_n;\boldsymbol{\Delta}) =
\{\tilde{\bst}_k\}_{k = 0}^{n - 1}$ is obtained by taking $\tilde{\bst}_k
\,=\, \left\{\bst_k+ \boldsymbol{\Delta}\right\}$, where the braces
indicate that we take the fractional part of each component in the vector
to ensure it still belongs to  $[0, 1]^s$. We will write the shifted QMC
approximation as $\Qsh{n,s}(P_n; \bsDelta)$.

In this setting, the \emph{shift-averaged worst-case error} is used as a
measure of the quality of a point set. It is simply the worst-case
error of the shifted point set, averaged in the root-mean-square
sense over all possible shifts, that is, $\wcesh{n, s, \bsgamma}(P_n)
\,\coloneqq\, (\int_{[0, 1]^s}e^2_{n,s, \bsgamma}(P_n; \bsDelta)
\,\rd\bsDelta)^{1/2}$. It then follows from the error bound
\eqref{eq:e*||f||} that the root-mean-square error of a shifted QMC
approximation (where the expected value is taken with respect to the shift
$\bsDelta$) satisfies
\begin{align}
\label{eq:rmsbnd}
\sqrt{\bbE\left(\left|I_sf - \Qsh{n,s}\left(P_n;  \cdot\right)f\right|^2\right)}
\,\leq\, \wcesh{n, s, \bsgamma}(P_n)\nrm{f}{s, \bsgamma} \, .
\end{align}
Again, the usefulness of this error bound lies in the fact that the right
hand side splits into two factors. Note that both factors depend on
the weights.

In practice we use a small number of independent and identically
distributed random shifts to estimate the error of approximation, see
e.g., \cite{DKS13}.

\subsection{Randomly shifted lattice rules}
\label{sec:lattice}

A \emph{rank-1 lattice rule} is a QMC rule for which the quadrature points
are generated by a single integer vector $\bsz$ called the
\emph{generating vector}. Each component $z_i$ belongs to $\U_n \coloneqq
\{z \in \bbN: z < n,\, \gcd(z, n) = 1\}$, the multiplicative group of
integers modulo $n$, and $\bsz \in \U_n^s = \U_n \times \dotsm \times
\U_n$. The number of positive integers less than and co-prime to $n$ is
given by the Euler totient function $\varphi(n) = |\U_n|$. So, for an
$n$-point lattice rule in $s$ dimensions there are
$\left(\varphi(n)\right)^s$ possible generating vectors.

As mentioned in the previous section, incorporating randomness into the
QMC rule is practically beneficial, and for lattice rules this is best
done by randomly shifting the points. Given some random shift
$\bsDelta\sim U\left([0, 1]^s\right)$, and generating vector $\bsz$, the
$\bsDelta$-shifted rank-1 lattice rule has points
\begin{align*} 
\tilde{\bst}_k \,=\, \left\{\frac{k\bsz}{n} + \boldsymbol{\Delta}\right\} \,
\end{align*}
for $k = 0, 1, \ldots, n - 1$.
The shift-averaged worst-case error of a randomly shifted lattice rule in
the space $\W_{s, \bsgamma}$ with general weights, is given explicitly by
(see \cite[Eq.\ (5.12)]{DKS13})
\begin{align}
\label{eq:wcelattice}
\wcesh{n, s, \bsgamma} (\bsz) \,=\,
\sqrt{\sum_{\emptyset \neq \setu \subseteq \{1:s\}} \gamma_\setu \left(\frac{1}{n} \sum_{k = 0}^{n - 1}
\prod_{j \in \setu} \Ber_2\left(\left\{\frac{kz_j}{n}\right\}\right)\right)} \,,
\end{align}
where $\Ber_2(x) = x^2 - x + \frac{1}{6}$ is the Bernoulli polynomial
of degree~$2$.

\subsection{The component-by-component construction}
\label{sec:cbc} The CBC algorithm is a method of constructing generating
vectors that results in ``good'' lattice rules in the context of
minimising the shift-averaged worst-case error \eqref{eq:wcelattice}. The
CBC construction is a greedy algorithm that works through each component
of the generating vector sequentially, choosing $z_{i}$ to \emph{minimise
the shift-averaged worst-case error in that dimension while all previous
components remain fixed}.

\begin{algorithm}[The CBC algorithm]\label{alg:cbc}\hfill\\
Given $n$ and $s$ and a sequence of weights $\bsgamma =
\{\gamma_\setu\}_{\setu\subseteq \{1:s\}}$.
\begin{enumerate}
\item Set $z_1$ to 1.
\item For $i = 2, \ldots, s$ choose $z_{i} \in \U_n$ so
    as to minimise $\wcesh{n, i, \bsgamma}(z_1, \ldots,
    z_{i - 1}, z_{i})$ given that all of the previous
    components $z_1, \ldots, z_{i - 1}$ remain fixed.
\end{enumerate}
\end{algorithm}

Setting $z_1$ to be 1 is done by convention, since in the first dimension
every choice results in an equivalent quadrature rule. Using structured
weights (product, order dependent or POD form) simplifies the formula for
the shift-averaged worst-case error \eqref{eq:wcelattice} even further,
allowing the calculation of $\wcesh{n, i, \bsgamma}(z_1, \ldots,
z_{i - 1}, z_{i})$ for all $z_{i} \in \U_n$ together
to be performed as one matrix-vector product. In general a naive
implementation of this algorithm costs $\calO(s\,n^2)$ operations, however
a fast construction performs the matrix-vector product using a fast Fourier transform (FFT)
reduces this to $\calO(s\,n\log n)$ in the case of product weights and
$\calO(s\,n\log n + s^2n)$ for order dependent or POD weights. For full
details on the \emph{Fast CBC construction} see \cite{DKS13, NC06,
NC06np}.

The shift-averaged worst-case error of a CBC generated lattice rule
satisfies the following upper bound:
\begin{align}
\label{eq:wcebnd}
\wcesh{n, s, \bsgamma}(\bsz) \,\leq\, \left(\frac{1}{\varphi(n)} \sum_{\emptyset \neq \setu \subseteq\{1:s\}}
\gamma_\setu^\lambda\left(\frac{2\zeta(2\lambda)}{(2\pi^2)^\lambda}\right)^{|\setu|}  \right)^\frac{1}{2\lambda}
\quad \text{for all } \lambda \in \left(\tfrac{1}{2}, 1\right] .
\end{align}
This result was proved in \cite{DKS13} for general weights and in
\cite{Dick04, Kuo03} for product weights.

\section{The CBC black box algorithms}{\label{sec:var}

The two new algorithms introduced in this section aim to choose weights so
as to make the bound \eqref{eq:rmsbnd} on the root-mean-square error
of the QMC approximation as small as possible. Hence we will require
a bound on the norm of the integrand $f \in \W_{s, \bsgamma}$ to be known
and of the specific form given in the Assumption \eqref{eq:L2bnd} in the Introduction.


In this way the norm of $f$ in $W_{s, \bsgamma}$, see
\eqref{eq:nrm_un}, with some, as yet unspecified, weights $\bsgamma$ will
be bounded by
\begin{align} \label{eq:nrmbnd}
\nrm{f}{s, \bsgamma}^2 \,\leq\, \sum_{\setu \subseteq \{1:s\}} \frac{1}{\gamma_\setu} B_{|\setu|} \prod_{j\in \setu} b_j^2 \,\eqqcolon\, M_{s, \bsgamma} \, ,
\end{align}
and in turn from \eqref{eq:rmsbnd} the mean-square error of a lattice rule
approximation will be bounded by
\begin{align}
\label{eq:msbnd_e*M}
\bbE\left(\left|I_sf - \Qsh{n, s}(\bsz; \cdot)f\right|^2\right) \,\leq\, \left(\wcesh{n, s, \bsgamma}(\bsz)\right)^2M_{s, \bsgamma} \, .
\end{align}

The first new algorithm is named the \textbf{double CBC (DCBC)
algorithm} since at each step two parameters are chosen: the component of
the generating vector and the  weight. We assume the weights are of
product or POD form. In the case of POD weights we assume that the order
dependent weight factors $\{\Gamma_\ell\}_{\ell = 0}^s$ are given.
Starting with the error bound \eqref{eq:msbnd_e*M}, in each
dimension, with all previous parameters remaining fixed, we
choose the component of $\bsz$ to minimise \wcesh{n, s, \bsgamma} and then
the product component of the weight to minimise the entire bound.

\subsection{The double CBC algorithm for product weights}\label{sec:dcbcpw}

In the case of product weights, the squared shift-averaged worst-case
error (see \eqref{eq:wcelattice}) of a lattice rule with generating
vector $\bsz$ is
\begin{align*} 
\left(\wcesh{n, s, \bsgamma}(z_1, \ldots, z_s)\right)^2 \,=\, &-1 + \frac{1}{n}\sum_{k = 0}^{n - 1} \prod_{j = 1}^s \left( 1 + \gamma_j
\Ber_2 \left( \left\{\frac{kz_j}{n}\right\}\right)\right) 
\\
\,=\, & \left(\wcesh{n, s-1, \bsgamma}(z_1, \ldots, z_{s - 1})\right)^2
+ \gamma_sG_s(z_1, \ldots, z_s)\,,
\end{align*}
in which the first term is independent of $\gamma_s$, and
\begin{align*} 
G_s(z_1, \ldots, z_s) \,:=\, \frac{1}{n}\sum_{k = 0}^{n - 1} \Ber_2 \left( \left\{\frac{kz_s}{n}\right\}\right)
\prod_{j = 1}^{s-1} \left(1 + \gamma_j
\Ber_2 \left( \left\{\frac{kz_j}{n}\right\}\right)\right)\, .
\end{align*}

For product weights it is natural to assume that the bound on the norm is
also of product form, that is, $B_\ell = 1$ for all $\ell = 1, 2, \ldots,
s$. It follows that this bound \eqref{eq:nrmbnd} can also be written
recursively as
\begin{align}
\label{eq:Mpw_rec}
M_{s, \bsgamma} \,=\, \sumu \prod_{j \in \setu}\frac{b_j^2}{\gamma_j}
\,=\, \prod_{j = 1}^s\left(1 + \frac{b_j^2}{\gamma_j}\right)
\,=\, \left(1 + \frac{b_s^2}{\gamma_s}\right)M_{s - 1, \bsgamma} \, .
\end{align}
In this situation, the bound \eqref{eq:msbnd_e*M} on the mean-square error
can be written as
\begin{align}
\label{eq:dcbcbndpw}
\left(\left(\wcesh{n, s - 1, \bsgamma}(z_1, \ldots, z_{s - 1})\right)^2
+ \gamma_s G_s(z_1, \ldots, z_s)\right)\left( 1 +
\frac{b_s^2}{\gamma_s}\right) M_{s - 1, \bsgamma}\, .
\end{align}

Treating \eqref{eq:dcbcbndpw} as a function of $\gamma_s$ and $z_s$, and
noting that $z_s$ is only present in $G_s$, at each step of the algorithm
we can first choose $z_s$ to minimise $G_s$ and then choose $\gamma_s$ to
minimise the entire error bound. For future reference, the minimiser of
expressions of this form is given by the following Lemma.

\begin{lemma}\label{lem:dcbcmin}
Suppose that $a, b, c, d$ are positive real numbers. Then the function
$h:(0, \infty) \rightarrow (0, \infty)$ given by $h(x) = (a + bx)(c +
\frac{d}{x}) $ is minimised by $x^* \,=\, \sqrt{\frac{ad}{bc}}$.
\end{lemma}
\begin{proof}
The first two derivatives of $h$ with respect to $x$ are $h^\prime(x) = bc
- ad/x^2$ and $h^{\prime\prime}(x) = 2ad/x^3 > 0$ for $x>0$, so $h$
is convex. Solving $h^\prime(x)  = 0$ yields the formula for the
stationary point $x^*$, which is the global minimum. \qquad
\end{proof}

Consequently, the choice of weight that minimises the bound on the
mean-square error \eqref{eq:dcbcbndpw} is given by, with $s$ replaced
by $i$,
\begin{align}
\label{eq:gamma_j_pw}
\gamma_{i} \,=\, \sqrt{\frac{\left(\wcesh{n, i - 1, \bsgamma}(z_1, \ldots, z_{i - 1})\right)^2
b_{i}^2}{G_{i}(z_1, \ldots, z_{i})}} \, .
\end{align}

Note that in the first dimension the upper bound on the mean-square error
\eqref{eq:dcbcbndpw} becomes $G_1\left(\gamma_1 + b_1^2\right)$, which
attains its minimum when $\gamma_1 = 0$. Since 0 is not a sensible choice
of weight our algorithm requires that $\gamma_1$ be given.

\begin{algorithm}[The double CBC algorithm for product weights]
\label{alg:dcbc_pw}
\hfill\\
Given $n$ and $s$, and bounds of the form \eqref{eq:L2bnd} with $B_\ell =
1$ for all $\ell$, and the weight in the first dimension $\gamma_1$,
set $z_1$ to 1. Then for each $i = 2, \dotsc, s$,
\begin{enumerate}
\item Choose  $z_{i} \in \U_n$ to minimise $G_{i}(z_1,
    \ldots, z_{i - 1}, z_{i})$.
\item Set $\gamma_i$ as in \eqref{eq:gamma_j_pw} and update the 
mean-square error bound \eqref{eq:dcbcbndpw}.
\end{enumerate}
\end{algorithm}

At each step of the algorithm, the process of choosing $z_{i}$ to minimise
$G_{i}$ is the same as in the original CBC algorithm. Thus, the
methods used in the fast CBC construction can also be applied in this
algorithm.

\subsection{The double CBC algorithm for POD weights}\label{sec:dcbcpod}

For weights of POD form, given a sequence of order dependent weight
factors $\{\Gamma_\ell\}$ and a bound on the norm $M_{s, \bsgamma}$ the
algorithm chooses the product component of the weights $\gamma_{i}$ in
each dimension. Note that, for this case we no longer assume all $B_\ell =
1$. As before, the first step is to obtain a recursive formula for the
bound on the norm of the integrand in each dimension. By splitting the sum
in \eqref{eq:nrmbnd} according to whether or not $s$ belongs to the
set $\setu$, we have
\begin{align}
\label{eq:Mpod_rec}
M_{s, \bsgamma}
\nonumber &= \, \sum_{\ell = 0}^s \frac{B_\ell}{\Gamma_\ell} \sum_{\substack{\setu \subseteq \{1:s\}\\
|\setu| = \ell}}\prod_{j \in \setu} \frac{b_j^2}{\gamma_j}
\,= \, \sum_{\ell = 0}^s \frac{B_\ell}{\Gamma_\ell} \left(\sum_{\substack{\setu \subseteq \{1:s-1\}\\ |\setu| = \ell}} \prod_{j\in \setu} \frac{b_j^2}{\gamma_j} + \sum_{\substack{s \in \setu \subseteq \{1:s\}\\ |\setu| = \ell}} \prod_{j \in \setu} \frac{b_j^2}{\gamma_j}\right)\\
&= \, M_{s-1, \bsgamma} + \frac{b_s^2}{\gamma_s} \sum_{\ell = 1}^s
\underbrace{\frac{B_\ell}{\Gamma_\ell}\sum_{\substack{\setu \subseteq \{1:s-1\}\\ |\setu| = \ell - 1}} \prod_{j \in \setu} \frac{b_j^2}{\gamma_j}}_{H_{s - 1, \ell - 1}}
\,= \, M_{s-1, \bsgamma} + \frac{b_s^2}{\gamma_s}\sum_{\ell = 0}^{s - 1}H_{s - 1, \ell}\,,
\end{align}
where we have introduced the terms $H_{i, \ell}$ to simplify the
notation. Applying a similar method of splitting the sum, a recursive
formula for $H_{s, \ell}$ is obtained
\begin{align}
\label{eq:Hsl}
 H_{s, \ell}
\nonumber\,&\coloneqq\, \frac{B_{\ell + 1}}{\Gamma_{\ell + 1}}\sum_{\substack{\setu \subseteq \{1:s\}\\ |\setu| = \ell}} \prod_{j \in \setu} \frac{b_j^2}{\gamma_j}
\,= \, \frac{B_{\ell + 1}}{\Gamma_{\ell + 1}}\sum_{\substack{\setu \subseteq \{1:s - 1\}\\ |\setu| = \ell}} \prod_{j \in \setu} \frac{b_j^2}{\gamma_j} + \frac{B_{\ell + 1}}{\Gamma_{\ell + 1}}\sum_{\substack{s \in \setu \subseteq \{1:s\}\\ |\setu| = \ell}} \prod_{j \in \setu} \frac{b_j^2}{\gamma_j}\\
\nonumber&=\, H_{s-1, \ell} + \frac{b_s^2B_{\ell + 1}}{\gamma_s\Gamma_{\ell + 1}}\sum_{\substack{\setu \subseteq \{1:s - 1\}\\ |\setu| = \ell - 1}} \prod_{j \in \setu} \frac{b_j^2}{\gamma_j}\\
&=\, H_{s-1, \ell} + \frac{b_s^2}{\gamma_s}\frac{B_{\ell + 1}}{B_\ell}\frac{\Gamma_\ell}{\Gamma_{\ell + 1}}H_{s - 1, \ell - 1} \, ,
\end{align}
with $H_{i, 0} = \frac{B_1}{\Gamma_1}$ for all $i = 1, 2,
\ldots, s$ and $H_{i,\ell} = 0$ for all $\ell > i$.

It follows that for POD weights the upper bound \eqref{eq:msbnd_e*M} on
the mean-square error of the QMC approximation can be written recursively
as
\begin{align}
\label{eq:dcbcbndpod}
&\bbE\left(\left|I_sf - \Qsh{n,s}(z_1, \ldots, z_s; \cdot )f\right|^2\right)\\
\nonumber \,\leq\, &\left(\left(\wcesh{n, s - 1, \bsgamma}(z_1, \ldots, z_s)\right)^2 + \gamma_s G_s(z_1, \ldots, z_s)\right)\left( M_{s - 1, \bsgamma}  + \frac{b_s^2}{\gamma_s}\sum_{\ell = 0}^{s - 1}H_{s - 1, \ell}\right)  \, ,
\end{align}
where $G_s(z_1, \ldots, z_s)$ is now given by
\begin{align}
\label{eq:Gs_ratio}
\nonumber &G_s(z_1, \ldots, z_s) \\
=\, &\frac{1}{n} \sum_{k = 0}^{n - 1} \Bigg(
\underbrace{\Ber_2\left(\left\{\frac{kz_s}{n}\right\}\right)}_{\Omega_n(z_s, k)}
\sum_{\ell = 1}^s\frac{\Gamma_\ell}{\Gamma_{\ell - 1}}
\underbrace{\sum_{\substack{\setu \subseteq \{1:s - 1\}\\|\setu| = \ell - 1}} \left(\prod_{i = 1}^{\ell - 1}\frac{\Gamma_i}{\Gamma_{i - 1}}\right)\produ \left(\gamma_j \Ber_2\left(\left\{\frac{kz_j}{n}\right\}\right)\right)}_{p_{s-1, \ell - 1}(k)}
\Bigg)\nonumber\\
=\, & \frac{1}{n} \sum_{k = 0}^{n - 1} \Omega_n(z_s, k) \sum_{\ell = 1}^s \frac{\Gamma_\ell}{\Gamma_{\ell - 1}} p_{s - 1, \ell - 1}(k) \, .
\end{align}
We have introduced the terms $\Omega_n$, $p_{s- 1, \ell - 1}$ to
simplify notation, and we have arranged to deal only with the ratios
$\Gamma_\ell/\Gamma_{\ell - 1}$ to improve numerical stability. Again by
Lemma~\ref{lem:dcbcmin}, the product component of the weight that
minimizes the bound on \eqref{eq:dcbcbndpod} is given by, with $s$
replaced by $i$,
\begin{align}
\label{eq:gamma_j_pod}
\gamma_{i} \,=\, \sqrt{\frac{\left(\wcesh{n, i, \bsgamma}(z_1, \ldots, z_{i - 1})\right)^2
b_{i}^2\sum_{\ell = 0}^{i - 1}H_{i - 1, \ell}}{M_{i - 1, \bsgamma}\,G_{i}(z_1, \ldots, z_{i})}} \, .
\end{align}

Calculating $G_s$ for all $z_s \in \U_n$ by summing over all $\setu
\subseteq \{1:s - 1\}$ as in \eqref{eq:Gs_ratio} would cost $\calO(n^22^{s
- 1})$ operations and is infeasible for even moderate $s$. The cost can be
reduced by storing $\Omega_n$ and constructing $p_{s - 1, \ell - 1}$
recursively. Letting $\boldsymbol{G}_s = \left[G_s(z_1, \ldots, z_{s - 1},
z_s)\right]_{z_s \in \U_n}$, the calculation of $G_s$ for all $z_s \in
\U_n$ can be performed by the matrix-vector product
\begin{align}
\label{eq:Gs_matrixvector}
\boldsymbol{G}_s \,=\, \frac{1}{n} \bsOmega_n \sum_{\ell = 1}^s \frac{\Gamma_\ell}{\Gamma_{\ell - 1}} \bsp_{s - 1, \ell - 1} \, ,
\end{align}
where
\begin{align*} 
\bsOmega_n \,\coloneqq\, \left[ \Ber_2\left(\left\{\frac{kz}{n}\right\}\right)\right]_{z \in \U_n, k = 0, 1, \ldots, n - 1}\,.
\end{align*}
At each step the vectors $p_{s - 1, \ell - 1}$ can be constructed
recursively as follows
\begin{align}
\label{eq:p_sl}
\bsp_{s, \ell}  \,=\, \bsp_{s - 1, \ell} + \frac{\Gamma_\ell}{\Gamma_{\ell - 1}}\gamma_s \bsOmega_n(z_s, :)
\,\dottimes\,\bsp_{s - 1, \ell - 1} \, ,
\end{align}
where $\bsOmega_n(z_s, :)$ is the row corresponding the new component of
the generating vector $z_s$, and $.*$ denotes component-wise
multiplication, and with $\bsp_{s, 0} = \bsone$, $\bsp_{s, \ell} =
\bszero$ for all $\ell > s$. Note that \eqref{eq:p_sl} is obtained by
splitting the sum according to whether or not $s \in \setu$, as in
\eqref{eq:Mpod_rec} and \eqref{eq:Hsl}. Since the cost of updating
$\bsp_{s, \ell}$ is $\calO(s\,n)$ the total cost of calculating $G_s$ in
each dimension has been reduced to $\calO(n^2 + sn)$ operations.
Additionally, using the concepts from the Fast CBC algorithm \cite{NC06,
NC06np} this product can be performed more efficiently using the FFT which
would further reduce the cost to $\calO(n\log n + sn)$. The total cost of
the algorithm is $\calO(s\,n\log n+ s^2n)$ operations.

\begin{algorithm}[The double CBC algorithm for POD weights]
\label{alg:dcbc_pod}
\hfill\\
Given $n$ and $s$, bounds of the form \eqref{eq:L2bnd}, order dependent
weight factors $\{\Gamma_\ell\}_{\ell = 0}^s$, and the weight in the
first dimension $\gamma_1$, set $z_1= 1$, $H_{0, 0}  = B_1/\Gamma_1$,
$\bsp_{0, 0} = \bsone$. Then for each $i = 2, \dotsc, s$,
\begin{enumerate}
\item For $\ell = 0, \ldots, i - 1$, update $H_{i - 1,
    \ell}$ using \eqref{eq:Hsl} and $\bsp_{i - 1, \ell}$ using
    \eqref{eq:p_sl}.
\item Calculate $\boldsymbol{G}_{i}$ using
    \eqref{eq:Gs_matrixvector} and FFT.
\item Choose  $z_{i} \in \U_n$ to minimise $G_{i}(z_1,
    \ldots, z_{i - 1}, z_{i})$.
\item Set $\gamma_i$ as in \eqref{eq:gamma_j_pod} and update the 
mean-square error bound \eqref{eq:dcbcbndpod}.
\end{enumerate}
\end{algorithm}

We have so far neglected the question of how to choose the order dependent
weight factors $\Gamma_\ell$. Three possible choices are:
\begin{itemize}
\item $\Gamma_\ell$ given \textit{a priori}, such as by the
    common choice $\Gamma_\ell = \ell!$.
\item $\Gamma_\ell = \Gamma_\ell(\lambda)$, that is, the order
    dependent weight factors of the weights
    $\gamma_\setu(\lambda)$ from the formula \eqref{eq:gammalambda}
    below. In this case we are still left with the predicament of how
    to choose $\lambda$, a choice which this algorithm aimed to
    circumvent.
\item $\Gamma_\ell = B_\ell$. Here the recursion for the bound on the
    norm \eqref{eq:Mpod_rec} is the same as the product weight
    case \eqref{eq:Mpw_rec}, that is, the terms $H_{i, \ell}$ are no
    longer required. Further, since there is some inherent connection
    between the form of the bound on the norm and the weights this
    choice seems more natural than the other two.
\end{itemize}

\subsection{The iterated CBC algorithm}\label{sec:icbc}

Combining \eqref{eq:msbnd_e*M} with the upper bound on shift-averaged
worst-case error \eqref{eq:wcebnd} we have that the mean-square error of a
CBC constructed lattice rule approximation is bounded above by
\begin{align*} 
\nonumber\bbE\left(\left|I_sf - \Qsh{n,s}(\bsz; \cdot)f\right|^2\right)
\,\leq\,
&\left(\frac{1}{\varphi(n)} \sum_{\emptyset \neq \setu \subseteq\{1:s\}}\gamma_\setu^\lambda\left(\frac{2\zeta(2\lambda)}{(2\pi^2)^\lambda}\right)^{|\setu|}  \right)^\frac{1}{\lambda} \\
& \times \left(\sum_{\setu \subseteq \{1:s\}}\frac{1}{\gamma_\setu}  B_{|\setu|} \prod_{j\in \setu} b_j^2\right)
\quad \text{for all } \lambda \in \left(\tfrac{1}{2}, 1\right] \, .
\end{align*}
From \cite[Lemma 6.2]{KSS12}, the weights that minimise this error bound
for each $\lambda \in \left(\frac{1}{2}, 1\right]$ are of POD form
\begin{align}
\label{eq:gammalambda}
\gamma_\setu(\lambda) \,=\, \Gamma_{|\setu|}(\lambda) \prod_{j \in \setu} \gamma_j(\lambda) \,=\, \left(B_{|\setu|} \prod_{j \in \setu} \frac{(2\pi^2)^\lambda b_j^2}{2\zeta(2\lambda)}\right)^\frac{1}{1 + \lambda} \, .
\end{align}

For each $\lambda \in \left(\frac{1}{2}, 1\right]$ the corresponding
weights $\bsgamma(\lambda) = \{\gamma_\setu(\lambda)\}_{\setu \subseteq
\{1:s\}}$ can be taken as input into the CBC algorithm to construct a
lattice rule generating vector which, through the weights, depends on 
$\lambda$: $\bsz(\lambda)$. Now that we know the weights
$\bsgamma(\lambda)$ and generating vector $\bsz(\lambda)$ explicitly, from
\eqref{eq:msbnd_e*M}, the mean-square error of the resulting QMC
approximation is bounded by
\begin{align}
\label{eq:gammalambdabnd}
 \bbE\left(\left|I_sf - Q^\mathrm{sh}_{s, n}(\bsz; \cdot)f\right|^2\right)
 \,\leq\, \left(\wcesh{n, s, \bsgamma(\lambda)}(\bsz(\lambda))\right)^2 M_{s, \bsgamma(\lambda)}
 \eqqcolon E_{n,s, \bsz(\lambda)}(\lambda)\, .
\end{align}

The goal of the \textbf{iterated CBC (ICBC) algorithm} is to carry out
iterations of the original CBC algorithm to choose a $\lambda$ that
minimises the right hand side of \eqref{eq:gammalambdabnd}. However, since
each component $z_j$ is obtained by a minimisation over a set of integers
and because this minimisation depends on the weights (and hence
$\lambda$), when treated as a function of $\lambda$ the shift-averaged
worst-case error is discontinuous. Hence, we cannot guarantee that a
minimum exists and as such our algorithm heuristically searches for a
``good'' value of $\lambda$. As mentioned in the introduction, the choice
of $\lambda$ is non-trivial since one needs to balance the size of the
constant and the theoretical convergence rate.

Suppose that in the upper bound \eqref{eq:gammalambdabnd} the generating
vector $\bsz$ remains fixed, then the upper bound, 
$E_{n,s,\bsz}(\lambda)$, is a continuous function of the single
variable $\lambda$ and can be minimised numerically.

The idea behind this algorithm is at each step of the iteration to use
$E_{n,s,\bsz^{(k)}}(\lambda)$ as an approximation to the right hand
side of the upper bound \eqref{eq:gammalambdabnd}. In this way the next
iterate $\lambda_{k + 1}$ is taken to be the minimiser of
$E_{n,s,\bsz^{(k)}}(\lambda)$, which can be found numerically using a
quasi-Newton method.

\begin{algorithm}[The iterated CBC algorithm]\label{alg:icbc}\hfill\\
Given $n$, $s$, bounds of the form \eqref{eq:L2bnd}, an initial
$\lambda_0 \in \left(\frac{1}{2}, 1\right]$, a tolerance $\tau$ and a
maximum number of iterations $k_\mathrm{max}$. For $k = 0, 1, 2, \ldots,
k_\mathrm{max}$:
\begin{enumerate}
\item Generate the weights $\gamma_\setu(\lambda_k)$ using \eqref{eq:gammalambda}.
\item Construct the generating vector $\bsz^{(k)}$ from the
    original CBC algorithm with weights
    $\gamma_\setu(\lambda_k)$.
\item If
    $\left|\dd{}{}{\lambda}E_{n,s,\bsz^{(k)}}(\lambda_{k})\right|
    < \tau$ then end the algorithm. \label{itm:stop}
\item Otherwise, choose $\lambda_{k + 1}$ to be the minimiser of
    $E_{n,s,\bsz^{(k)}}(\lambda)$, found numerically using a
    quasi-Newton algorithm.\label{itm:quasiNewton}
\end{enumerate}
\end{algorithm}

\begin{remark}
\label{rem:dE} For the quasi-Newton algorithm in
Step~\ref{itm:quasiNewton} we require the derivative of
$E_{n,s,\bsz}$, for fixed $\bsz$, with respect to $\lambda$
\begin{align*} 
\nonumber &\dd{}{E_{n,s,\bsz}}{\lambda}
=
\nonumber \left(\sum_{\emptyset \neq \setu \subseteq \{1:s\}} \gamma^\prime_\setu(\lambda) \left(\frac{1}{n} \sum_{k = 0}^{n - 1} \prod_{j \in \setu} \Ber_2\left(\left\{\frac{kz_j}{n}\right\}\right)\right)\right)
\left(\sum_{\setu \subseteq \{1:s\}} \frac{B_{|\setu|} \prod_{j\in \setu} b_j^2}{\gamma_\setu(\lambda)} \right)\\
&-\left(\sum_{\emptyset \neq \setu \subseteq \{1:s\}} \gamma_\setu(\lambda) \left(\frac{1}{n} \sum_{k = 0}^{n - 1} \prod_{j \in \setu} \Ber_2\left(\left\{\frac{kz_j}{n}\right\}\right)\right)\right)
\left(\sum_{\setu \subseteq \{1:s\}} \frac{\gamma^\prime_\setu(\lambda)}{\gamma^2_\setu(\lambda)} B_{|\setu|} \prod_{j\in \setu} b_j^2\right) \, ,
\end{align*}
where the derivative of each weight with respect $\lambda$ is
\begin{align*} 
\gamma^\prime_\setu(\lambda)
 \,=\,
\gamma_\setu(\lambda)
\left(\frac{-\log\left(B_{|\setu|}\prod_{j \in \setu}\frac{\left(2\pi^2\right)^\lambda b_j^2}{2\zeta(2\lambda)}\right)}{(1 + \lambda)^2} + \frac{|\setu|\left(\log(2\pi^2) - \frac{2\zeta^\prime(2\lambda)}{\zeta(2\lambda)}\right)}{(1 + \lambda)}\right)\, .
\end{align*}
\end{remark}

\section{Numerical results}\label{sec:num_cbc}

For the numerical results we look at how each new method performs for
different types of bounds \eqref{eq:L2bnd}, that is, for different sequences
$\bsb \coloneqq (b_i)_{i = 1}^s$ and $\bsB \coloneqq (B_\ell)_{\ell = 1}^s$. 
As a figure of merit we will use the upper bound on the 
root-mean-square (RMS) error (cf. \eqref{eq:msbnd_e*M}), which we denote by 
 \begin{align*}
 E_{n, s, \bsb, \bsB}(\bsgamma, \bsz) \,\coloneqq\, \wcesh{n, s, \bsgamma}(\bsz)\sqrt{M_{s, \bsgamma}}\,.
 \end{align*}
Here we have specifically used this notation to indicate the dependence 
on the sequences $\bsb,\,\bsB$ but also to emphasise that this upper bound is 
primarily a function of $\bsgamma$ and $\bsz$, the outputs of our algorithms.
Note also that in Tables \ref{tab:pw_j^-2}--\ref{tab:pod_fact_0.5^j} the results given
for our algorithms are \emph{guaranteed} error bounds (in the RMS sense) for integrands 
which satisfy the appropriate bounds.

In the examples let the maximum dimension be $s = 100$ and the number 
of points $n$ be prime and ranging up to 32,003. Here we choose 
$n$ to be prime because it makes the ``fast'' aspects of the implementation 
simpler, but note that $n$ prime is not a requirement of either algorithm. 
Also, we use the notation ``e'' for the base-10 exponent.

\subsection{The case $B_\ell = 1$}

As a start, let $B_\ell = 1$ for all $\ell = 1, \ldots, s$, and
consider the cases $b_i= i^{-2}$, $b_i = 0.5^i$ and $b_i = 0.8^i$. 
With $B_\ell = 1$ it is natural to restrict attention to product weights. The
results for this case are given in Tables~\ref{tab:pw_j^-2}, \ref{tab:pw_0.5^j},
\ref{tab:pw_0.8^j}, respectively. These tables compare results for
$E_{n, s, \bsb, \bsB}$ from the DCBC, ICBC algorithms with the original CBC
algorithm using common choices of product weights. The choices of
common weights are $\gamma_i = i^{-1.1}$, $\gamma_i = i^{-2}$ and
$\gamma_i(\lambda)$ as in \eqref{eq:gammalambda} with $\lambda = 0.6, 1$.
The row labelled ``rate'' gives the exponent ($x$) for a least-squares fit
of the result $E_{n, s, \bsb, \bsB}$ to a power law ($n^{-x}$).
\begin{table}[p]
\centering
\begin{tabular}{r|c|c||c|c|c|c}
& \multicolumn{2}{|c||}{Variants} & \multicolumn{4}{|c}{CBC with common weights}\\
\hline
$n$ & \textbf{DCBC} & ICBC & $\gamma_i = i^{-1.1}$ & $\gamma_i = i^{-2}$ & $\gamma_i(\lambda = 0.6)$ & $\gamma_i(\lambda = 1)$\\
\hline
251 & \textbf{6.8e-3} & 7.0e-3 & 3.5e-2 & 7.5e-3 & 8.2e-3 &  1.3e-2 \\
499 & \textbf{3.5e-3} & 3.6e-3 & 2.1e-2 & 4.0e-3 & 4.2e-3 & 7.6e-3 \\
997 & \textbf{1.8e-3} & 1.9e-3 & 1.3e-2 & 2.2e-3 & 2.2e-3 & 4.3e-3 \\
1999 & \textbf{9.7e-4} & 1.0e-3 & 7.8e-3 & 1.2e-3 & 1.1e-3 & 2.4e-3 \\
4001 & \textbf{5.1e-4} & 5.2e-4 & 4.8e-3 & 6.3e-4 & 5.8e-4 & 1.4e-3\\
7993 &  \textbf{2.7e-4} & 2.7e-4 & 2.9e-3 & 3.4e-4 & 2.9e-4 & 7.8e-4\\
16001 & \textbf{1.4e-4} & 1.4e-4 & 1.8e-3 & 1.9e-4 & 1.5e-4 & 4.4e-4\\
32003 & \textbf{7.4e-5} & 7.5e-5 & 1.1e-3 & 1.0e-4 & 7.9e-5 & 2.5e-4\\
\hline
rate & \textbf{0.93} & 0.93 & 0.71 & 0.88 & 0.95 & 0.82
\end{tabular}
\vspace*{-0.3cm}\caption{Results for the root-mean-square error bound $E_{n, s, \bsb, \bsB}$ for $b_i = i^{-2}$: DCBC, ICBC and CBC results for common choices of weights.}
\label{tab:pw_j^-2}
\end{table}
\begin{table}
\centering
\begin{tabular}{r|c|c||c|c|c|c}
& \multicolumn{2}{|c||}{Variants} & \multicolumn{4}{|c}{CBC with common weights}\\
\hline
$n$ & DCBC & \textbf{ICBC} & $\gamma_i = i^{-1.1}$ & $\gamma_i = i^{-2}$ & $\gamma_i(\lambda = 0.6)$ &  $\gamma_i(\lambda = 1)$\\
\hline
251 & 4.1e-3 &\textbf{ 3.3e-3} &  2.8e-2 & 5.5e-3 & 3.3e-3 &  6.7e-3 \\
499 & 2.1e-3 & \textbf{1.7e-3} & 1.7e-2 & 2.9e-3 & 1.7e-3 & 3.6e-3 \\
997 & 1.1e-3 & \textbf{8.6e-4} & 1.0e-2 & 1.6e-3 & 8.6e-4 & 2.0e-3 \\
1999 & 5.6e-4 & \textbf{4.4e-4} & 6.2e-3 & 8.6e-4 & 4.4e-4 & 1.1e-3 \\
4001 & 2.9e-4 & \textbf{2.2e-4} & 3.8e-3 & 4.6e-4 & 2.2e-4 & 5.8e-4\\
7993 &1.5e-4 & \textbf{1.1e-4} & 2.3e-3 & 2.5e-4 & 1.1e-4 & 3.1e-4\\
16001 & 7.6e-5 & \textbf{5.9e-5} & 1.4e-3 & 1.4e-4 & 5.9e-5 & 1.7e-4 \\
32003 & 3.9e-5 & \textbf{3.0e-5} & 8.7e-4 & 7.5e-5 & 3.0e-5 & 9.3e-5 \\
\hline
rate & 0.96 & \textbf{0.96} & 0.71 & 0.88 &  0.97 & 0.88
\end{tabular}
\vspace*{-0.3cm}\caption{Results for the root-mean-square error bound $E_{n, s, \bsb, \bsB}$ for $b_i = 0.5^i$:  DCBC, ICBC and CBC results for common choices of weights.}
\label{tab:pw_0.5^j}
\end{table}
\begin{table}
\centering
\begin{tabular}{r|c|c||c|c|c|c}
& \multicolumn{2}{|c||}{Variants} & \multicolumn{4}{|c}{CBC with common weights}\\
\hline
$n$ & DCBC & \textbf{ICBC} & $\gamma_i = i^{-1.1}$ & $\gamma_i = i^{-2}$ & $\gamma_i(\lambda = 0.6)$ & $\gamma_i(\lambda = 1)$\\
\hline
251 & 9.9e-2 & \textbf{8.3e-2} & 2.0e-1 & 2.8 & 1.6e-1  & 1.2e-1 \\
499 & 5.7e-2 & \textbf{5.0e-2} & 1.2e-1 & 1.5 & 8.9e-2  & 7.2e-2 \\
997 & 3.5e-2 & \textbf{2.9e-2} & 7.5e-2 & 8.2e-1 & 5.1e-2 & 4.5e-2 \\
1999 & 2.1e-2 & \textbf{1.7e-2} & 4.6e-2 & 4.4e-1 & 2.8e-2 & 2.8e-2 \\
4001 & 1.2e-2 & \textbf{1.0e-2} & 2.8e-2 & 2.4e-1 & 1.6e-2 & 1.8e-2 \\
7993 & 7.3e-3 & \textbf{5.9e-3} & 1.7e-2 & 1.3e-1 & 9.1e-3 & 1.1e-2 \\
16001 & 4.3e-3 & \textbf{3.5e-3} & 1.0e-2 & 7.1e-2 & 5.0e-3 & 6.7e-3 \\
32003 & 2.5e-3 & \textbf{2.0e-3} & 6.4e-3 & 3.9e-2 & 2.9e-3 & 4.2e-3 \\
\hline
rate & 0.75 & \textbf{0.75} & 0.71 & 0.88 & 0.82 & 0.69
\end{tabular}
\vspace*{-0.3cm}\caption{Results for the root-mean-square error bound $E_{n, s, \bsb, \bsB}$ for $b_i = 0.8^i$: DCBC, ICBC and CBC results for common choices of weights.}
\label{tab:pw_0.8^j}
\end{table}

Comparing results for the two new algorithms, note that for $b_i = i^{-2}$
(Table~\ref{tab:pw_j^-2}) the results of the DCBC algorithm are
better, while for $b_i = 0.5^i$, $0.8^i$ (Tables~\ref{tab:pw_0.5^j},
\ref{tab:pw_0.8^j}) the ICBC algorithm produces better bounds, and so it
is not the case that one algorithm is always better than the other. However,
in all cases the ICBC algorithm performs as well as or better than
the original CBC algorithm with common choices of weights.

Table~\ref{tab:pw_icbc} gives the final value  of $\lambda$, denoted $\lambda^*$,
resulting from the ICBC algorithm for our three choices of $\bsb$, along
with the resulting RMS error bound. Notice that, as expected the value of
$\lambda^*$ found by the algorithm appears to approach 0.5 as $n$ increases,
albeit very slowly.
\begin{table} [p]
\centering
\begin{tabular}{r||c|c|c}
$n$ & $b_i = i^{-2}$ & $b_i = 0.5^i$ &$b_i = 0.8^i$\\
\hline
251 & 0.672 & 0.616 & 0.756 \\
499 & 0.668 & 0.615 & 0.744 \\
997 & 0.661 & 0.610  & 0.735 \\
1999 & 0.657 & 0.607 & 0.725 \\
4001 & 0.652 & 0.604  & 0.715 \\
7993 & 0.645 & 0.601 & 0.711 \\
16001 & 0.642 & 0.597  & 0.700 \\
32003 &  0.637 & 0.594 & 0.696  \\
\end{tabular}
\vspace*{-0.3cm}\caption{Value of $\lambda^*$ from ICBC for product weights with $b_i = i^{-2}$, $0.5^i$ and $0.8^i$.}
\label{tab:pw_icbc}
\end{table}

\begin{figure} 
\hspace*{-1cm}
\begin{subfigure}[b]{0.49\textwidth}
\includegraphics[scale=0.4]{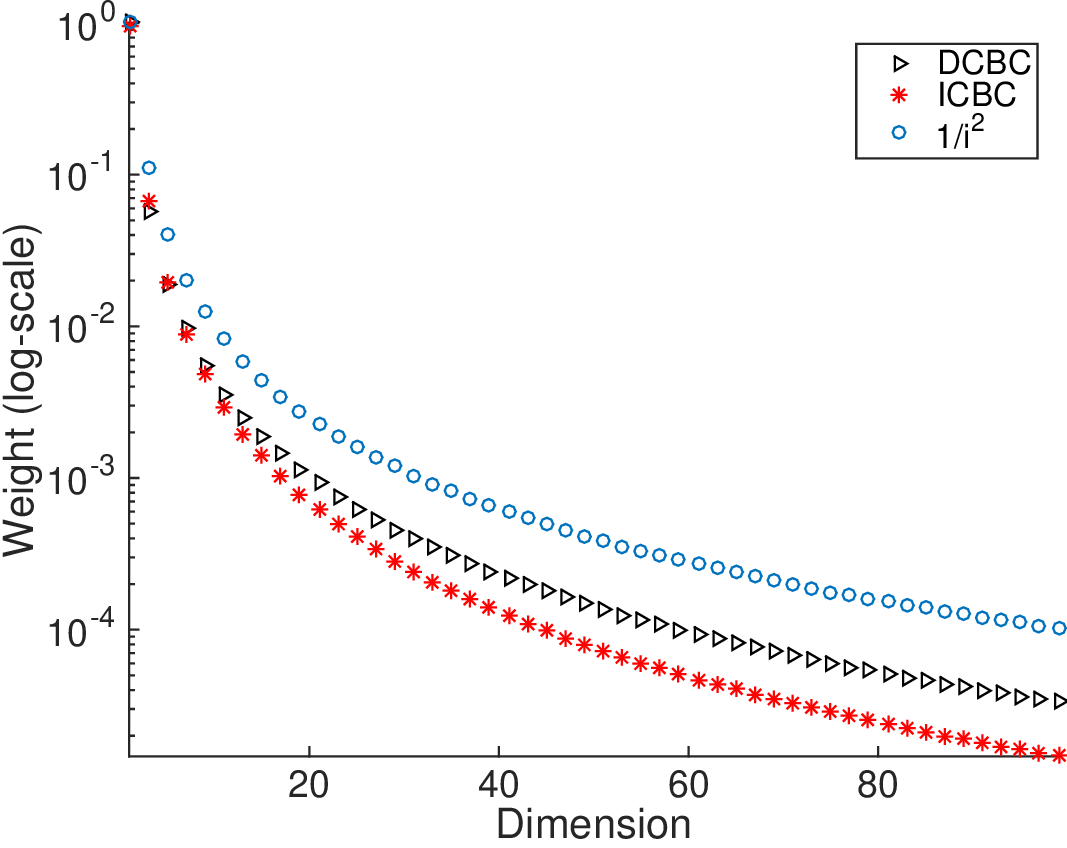}
\vspace*{-0.3cm}\caption{$b_i = i^{-2}$.}
\label{fig:weight_j^2}
\end{subfigure}
\hspace*{1cm}
\begin{subfigure}[b]{0.49\textwidth}
\includegraphics[scale=0.4]{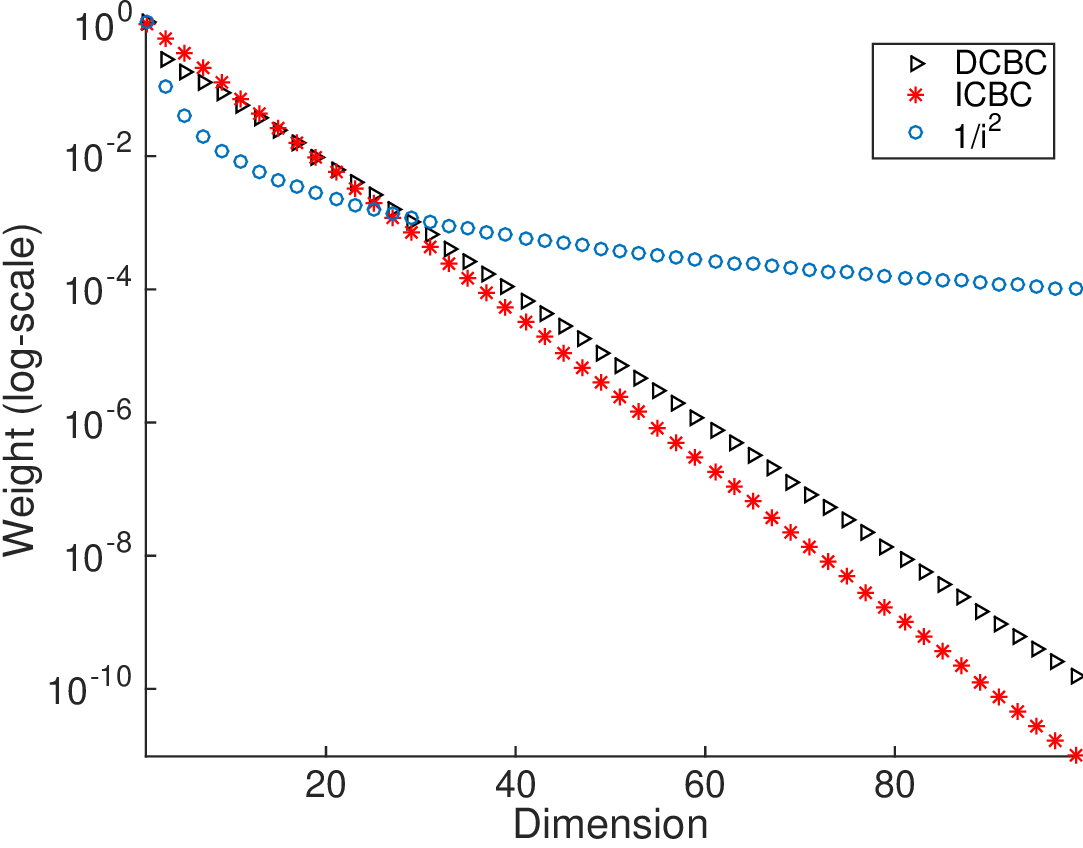}
\vspace*{-0.3cm}\caption{$b_i = 0.5^{i}$.}
\label{fig:weight_2^-j}
\end{subfigure}
\vspace*{-0.3cm}\caption{Weight in each dimension found from DCBC and ICBC, with $\gamma_i = i^{-2}$ for comparison ($n = 1999$).}
\end{figure}

Figures~\ref{fig:weight_j^2} and \ref{fig:weight_2^-j} compare the
weight $\gamma_i$ in each dimension for the DCBC, ICBC algorithms,
with $\gamma_i = i^{-2}$ as a reference, for $b_i = i^{-2}$ and $b_i =
0.5^i$, respectively. Here $n = 1999$.
%

Figure~\ref{fig:errdim} compares the RMS error bound in each dimension for
the DCBC and ICBC algorithms, for $b_i = 0.8^i$ with $n = 1999$.
Notice that due to the greedy nature of the DCBC algorithm, it
performs better in the earlier dimensions, however, the ICBC
algorithm produces weights with a better bound overall.
\begin{figure} 
\centering
\includegraphics[scale=0.4]{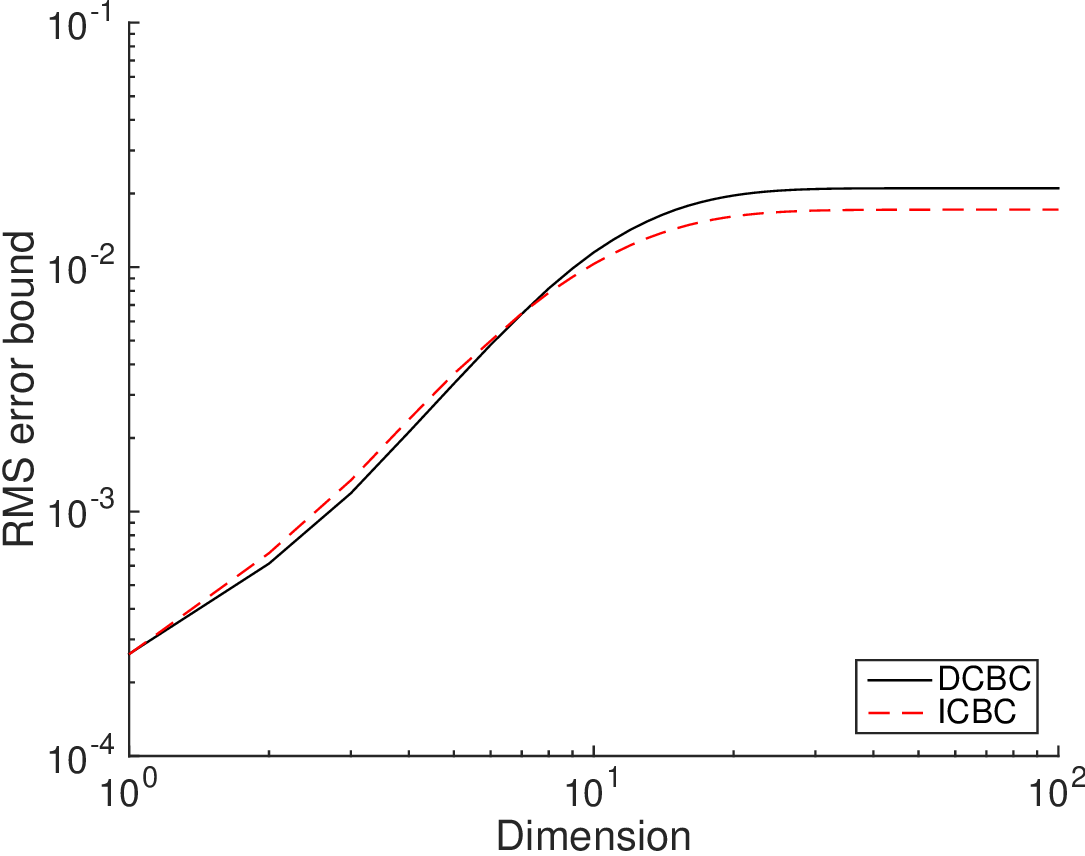} 
\caption{$E_{n, s, \bsb, \bsB}$ in each dimension for DCBC and ICBC with $b_i = 0.8^i$ (loglog scale).}
\label{fig:errdim}
\end{figure}

\subsection{POD weights}

Now, we look at the performance of each algorithm for combinations of
$B_\ell = \ell$ and $B_\ell = \ell!$  with $b_i = i^{-2}$ and $b_i =
0.5^i$, the results are given in
Tables~\ref{tab:pod_ell_j^-2}--\ref{tab:pod_fact_0.5^j}. Each table
presents the RMS error bound $E_{n, s, \bsb, \bsB}$ obtained from the DCBC
algorithm for different choices of order dependent weights, and
those obtained from the ICBC algorithm along with the output $\lambda^*$.
With regards to the strategy for choosing $\Gamma_\ell$ in the DCBC
algorithm for POD weights, in three out of the four cases the best
choice was to let $\Gamma_\ell = B_\ell$ (see
Tables~\ref{tab:pod_ell_j^-2}--\ref{tab:pod_ell_0.5^j}). However, in all
cases the results are similar and indicate that the choice of
$\Gamma_\ell$ does not greatly effect the final bound $E_{n, s, \bsb, \bsB}$.
\begin{table} [!h]
\centering
\begin{tabular}{r||c|c||c|c}
& \multicolumn{2}{c||}{DCBC} & \multicolumn{2}{|c}{\textbf{ICBC}}\\
\hline
$n$ & $E_{n, s, \bsb, \bsB}$ ($\Gamma_\ell = B_\ell$) & $E_{n, s, \bsb, \bsB}$ ($\Gamma_\ell = \ell!$) & \textbf{$E_{n, s, \bsb, \bsB}$} & $\lambda^*$\\
\hline
251 & 8.6e-3 & 8.5e-3 & \textbf{8.7e-3} & 0.680 \\
499 & 4.6e-3 & 4.5e-3 & \textbf{4.6e-3} & 0.673 \\
997 & 2.5e-3 & 2.5e-3 & \textbf{2.5e-3} & 0.666 \\
1999 & 1.3e-3 & 1.3e-3 & \textbf{1.3e-3} & 0.659 \\
4001 & 6.9e-4 & 7.0e-4 & \textbf{6.8e-4} & 0.655\\
7993 & 3.7e-4 & 3.7e-4 & \textbf{3.6e-4} & 0.650 \\
16001 & 1.9e-4 & 2.0e-4 & \textbf{1.9e-4} & 0.645 \\
32003 & 1.0e-4 & 1.1e-4 & \textbf{1.0e-4} & 0.640 \\
\hline
rate & 0.91 & 0.90 & \textbf{0.92} &
\end{tabular}
\vspace*{-0.3cm}\caption{POD weight results with $b_i= i^{-2}$ and $B_\ell = \ell$: $E_{n, s, \bsb, \bsB}$ from DCBC for different choices of $\Gamma_\ell$, $E_{n, s, \bsb, \bsB}$ from ICBC and $\lambda^*$ from ICBC.}
\label{tab:pod_ell_j^-2}
\end{table}
\begin{table} [!h]
\centering
\begin{tabular}{r||c|c||c|c}
& \multicolumn{2}{c||}{\textbf{DCBC}} & \multicolumn{2}{|c}{ICBC}\\
\hline
$n$ & $E_{n, s, \bsb, \bsB}$ ($\Gamma_\ell = B_\ell$) & $E_{n, s, \bsb, \bsB}$ ($\Gamma_\ell = \ell$) & $E_{n, s, \bsb, \bsB}$ & $\lambda^*$\\
\hline
251 & \textbf{9.2e-3} & 1.1e-2 & 9.7e-3 & 0.692 \\
499 & \textbf{5.0e-3} & 5.8e-3 & 5.1e-3 & 0.685 \\
997 & \textbf{2.7e-3} & 3.2e-3 & 2.8e-3 & 0.679 \\
1999 & \textbf{1.5e-3} & 1.7e-3 & 1.5e-3 & 0.673 \\
4001 & \textbf{7.9e-4} & 9.6e-4 & 8.0e-4 & 0.667 \\
7993 & \textbf{4.2e-4} & 5.2e-4 & 4.3e-4 & 0.661 \\
16001 & \textbf{2.3e-4} & 2.8e-4 & 2.3e-4 & 0.656 \\
32003 & \textbf{1.2e-4} & 1.6e-4 & 1.3e-4 & 0.651 \\
\hline
rate & \textbf{0.89} & 0.87 & 0.89 &
\end{tabular}
\vspace*{-0.3cm}\caption{POD weight results with $b_i= i^{-2}$ and $B_\ell = \ell!$: $E_{n, s, \bsb, \bsB}$ from DCBC for different choices of $\Gamma_\ell$, $E_{n, s, \bsb, \bsB}$ from ICBC and $\lambda^*$ from ICBC.}
\label{tab:pod_fact_j^-2}
\end{table}
\begin{table} [!h]
\centering
\begin{tabular}{r||c|c||c|c}
& \multicolumn{2}{c||}{DCBC} & \multicolumn{2}{|c}{\textbf{ICBC}}\\
\hline
$n$ & $E_{n, s, \bsb, \bsB}$ ($\Gamma_\ell = B_\ell$) & $E_{n, s, \bsb, \bsB}$ ($\Gamma_\ell = \ell!$) & $E_{n, s, \bsb, \bsB}$ & $\lambda^*$\\
\hline
251 &  4.9e-3 & 5.0e-3 & \textbf{3.8e-3} & 0.619 \\
499 & 2.5e-3 &  2.6e-3 & \textbf{2.0e-3} & 0.617 \\
997 & 1.3e-3 & 1.4e-3 & \textbf{1.0e-3} & 0.612 \\
1999 & 6.9e-4 & 7.2e-4 & \textbf{5.3e-4} & 0.608 \\
4001 & 3.6e-4 & 3.8e-4 & \textbf{2.7e-4} & 0.605 \\
7993 & 1.9e-4 & 2.0e-4 & \textbf{1.4e-4} & 0.602 \\
16001 & 9.8e-5 & 1.0e-4 & \textbf{7.2e-5} & 0.597 \\
32003 & 5.1e-5 & 5.3e-5 & \textbf{3.7e-5} & 0.595 \\
\hline
rate & 0.94 & 0.93 & \textbf{0.95} &
\end{tabular}
\vspace*{-0.3cm}\caption{POD weight results with $b_i= 0.5^i$ and $B_\ell = \ell$: $E_{n, s, \bsb, \bsB}$ from DCBC for different choices of $\Gamma_\ell$, $E_{n, s, \bsb, \bsB}$ from ICBC and $\lambda^*$ from ICBC.}
\label{tab:pod_ell_0.5^j}
\end{table}
\begin{table} [!h]
\centering
\begin{tabular}{r||c|c||c|c}
& \multicolumn{2}{c||}{DCBC} & \multicolumn{2}{|c}{\textbf{ICBC}}\\
\hline
$n$ & $E_{n, s, \bsb, \bsB}$ ($\Gamma_\ell = B_\ell$) & $E_{n, s, \bsb, \bsB}$ ($\Gamma_\ell = \ell$) & $E_{n, s, \bsb, \bsB}$ & $\lambda^*$\\
\hline
251 &  5.1e-3 & 5.1e-3 & \textbf{4.0e-3} & 0.625 \\
499 & 2.6e-3 & 2.6e-3 & \textbf{2.1e-3} & 0.622 \\
997 & 1.4e-3 & 1.4e-3 & \textbf{1.1e-3} & 0.618 \\
1999 & 7.3e-4 & 7.3e-4 & \textbf{5.6e-4} & 0.614 \\
4001 & 3.9e-4 & 3.8e-4 & \textbf{2.9e-4} & 0.608 \\
7993 & 2.0e-4 & 2.0e-4 & \textbf{1.5e-4} & 0.604 \\
16001 & 1.1e-4 & 1.0e-4 & \textbf{7.9e-5} & 0.602 \\
32003 & 5.6e-5 & 5.5e-5 & \textbf{4.1e-5} & 0.599 \\
\hline
rate & 0.93 & 0.93 & \textbf{0.95} &
\end{tabular}
\vspace*{-0.3cm}\caption{POD weight results with $b_i= 0.5^i$ and $B_\ell = \ell!$: $E_{n, s, \bsb, \bsB}$ from DCBC for different choices of $\Gamma_\ell$, $E_{n, s, \bsb, \bsB}$ from ICBC and $\lambda^*$ from ICBC.}
\label{tab:pod_fact_0.5^j}
\end{table}

\section{Concluding remark}

We introduced two new CBC algorithms, the double CBC (DCBC)
algorithm and the iterated CBC (ICBC) algorithm, which only require
parameters specified by the problem to determine the point set for
QMC integration, and provide guaranteed error bounds by also choosing ``good'' weight 
parameters. The numerical results show different examples where each algorithm 
performs better than the other. Both algorithms generally outperform the original CBC 
algorithm with common choices of weights. In all cases the entries $E_{n, s, \bsb, \bsB}$
provide guaranteed upper bounds on the root-mean-square error for the randomly shifted
integration rules, under the indicated assumptions on the bound parameters $B_\ell$ and $b_i$.


\bibliographystyle{plain}

\noindent
Alexander D. Gilbert\\
alexander.gilbert@unsw.edu.au\\
School of Mathematics and Statistics, University of New South Wales, Sydney NSW 2052, Australia

\smallskip
\noindent
Frances Y. Kuo\\
f.kuo@unsw.edu.au\\
School of Mathematics and Statistics, University of New South Wales, Sydney NSW 2052, Australia

\smallskip
\noindent
Ian H. Sloan\\
i.sloan@unsw.edu.au\\
School of Mathematics and Statistics, University of New South Wales, Sydney NSW 2052, Australia

\end{document}